\documentclass[11pt, a4paper]{article}

\usepackage{times}
\usepackage{a4wide}
\usepackage{hyperref}
\usepackage[british]{babel}
\usepackage{enumerate}
\usepackage{amsmath, amscd, amsfonts, amsthm, amssymb, latexsym, comment, stmaryrd, graphicx}
\usepackage[all]{xypic}
\usepackage[T1]{fontenc}
\usepackage[utf8]{inputenc}

\newtheorem{theorem}{Theorem}
\newtheorem{lemma}[theorem]{Lemma}
\newtheorem{definition}[theorem]{Definition}

\newtheorem{proposition}[theorem]{Proposition}
\newtheorem{corollary}[theorem]{Corollary}

\newcommand{\CC}{\mathbb{C}}
\newcommand{\FF}{\mathbb{F}}
\newcommand{\NN}{\mathbb{N}}
\newcommand{\QQ}{\mathbb{Q}}
\newcommand{\TT}{\mathbb{T}}
\newcommand{\ZZ}{\mathbb{Z}}

\newcommand{\Qbar}{\overline{\QQ}}
\newcommand{\Zbar}{\overline{\ZZ}}
\newcommand{\Fbar}{\overline{\FF}}
\newcommand{\Tbar}{\overline{\TT}}

\newcommand{\cO}{\mathcal{O}}

\newcommand{\fa}{\mathfrak{a}}
\newcommand{\fm}{\mathfrak{m}}
\newcommand{\fn}{\mathfrak{n}}

\DeclareMathOperator{\Gal}{Gal}
\DeclareMathOperator{\Frob}{Frob}
\DeclareMathOperator{\Tr}{Tr}
\DeclareMathOperator{\Spec}{Spec}

\newcommand{\GL}{\mathrm{GL}}
\newcommand{\SL}{\mathrm{SL}}
\newcommand{\rhobar}{{\overline{\rho}}}
\newcommand{\Hom}{{\rm Hom}}
\newcommand{\End}{{\rm End}}
\newcommand{\fmat}[4]{\begin{pmatrix} #1 & #2 \\ #3 & #4 \end{pmatrix}}           
\newcommand{\fvect}[2]{\begin{pmatrix} #1 \\ #2 \end{pmatrix}}                    
\newcommand{\Zmod}[1]{\overline{\ZZ/{#1}\ZZ}}
\newcommand{\pf}{\textrm{pf}}
\newcommand{\la}{\lambda}
\newcommand{\Sh}{\mathrm{Sh}}
\newcommand{\isom}{\cong}
\newcommand{\Supp}{\operatorname{Supp}}
\newcommand{\al}{\alpha}
\newcommand{\ram}{\textrm{ram}}
\newcommand{\unram}{\textrm{unr}}

\begin{document}

\title{Topics on modular Galois representations modulo prime powers}
\author{Panagiotis Tsaknias and Gabor Wiese\footnote{Universit\'e du Luxembourg,
Unit\'e de Recherche en Math\'ematiques,
Maison du nombre,
6, avenue de la Fonte,
L-4364 Esch-sur-Alzette, Luxembourg; p.tsaknias@gmail.com, gabor.wiese@uni.lu }}

\maketitle

\abstract{
This article surveys modularity, level raising and level lowering questions for two-dimensional representations modulo prime powers
of the absolute Galois group of the rational numbers.
It contributes some new results and describes algorithms and a database of modular forms orbits and higher congruences.\\
MS Classification (2010): 11F33, 11F80
}

\section{Introduction}\label{sec:intro}

The Fontaine-Mazur conjecture relates $\ell$-adic `geometric' Galois representations with objects from geometry.
In the $2$-dimensional case over~$\QQ$ much progress has been achieved (\cite[Thm.1.2.4(2)]{EmertonLG} and \cite{KisinFM}):
\begin{theorem}[Emerton, Kisin]
Let $\ell > 2$, let $E/\QQ_\ell$ be a finite extension and let $\rho: \Gal(\Qbar/\QQ) \to \GL_2(E)$ be an irreducible, finitely ramified,
odd Galois representation which is de Rham at $\ell$ with distinct Hodge-Tate weights.
Assume that the residual representation $\rhobar$ satisfies certain local conditions.

Then a twist of $\rho$ is attached to some newform.
\end{theorem}
In fact, the level and the weight of the newform can be read off from~$\rho$.

The picture for mod $\ell$ representations is even more complete:
Serre type modularity conjectures relate $2$-dimensional Galois representations with $\Fbar_\ell$-coefficients with modular forms over~$\Fbar_\ell$.
Serre's original modularity conjecture has been established by Khare and Wintenberger~\cite{KW}:
\begin{theorem}[Khare, Wintenberger, Kisin]
Let $\rhobar: \Gal(\Qbar/\QQ) \to \GL_2(\Fbar_\ell)$ be an odd irreducible Galois representation.
Then $\rhobar$ is attached to (the reduction of) some newform.
\end{theorem}
Also in this case, the level and the weight of one of the (infinitely many) newforms attached to~$\rhobar$ can be read off from~$\rhobar$:
the level is the prime-to-$\ell$ Artin conductor, and there is a formula for the weight given by Serre.
In fact, it had been known for a long time that if $\rhobar$ is attached to {\em some} newform then it also is to
one with a certain predicted weight and level. Since the predicted weight and level are minimal
(except for two cases in the weight, see \cite{Edixhoven1992}), the process of finding a newform with predicted invariants
is called {\em level lowering} or {\em weight lowering}.
The quest for attached newforms with non-minimal levels is accordingly called {\em level raising}.
These three questions have been completely solved (with a tiny exception when $\ell=2$ and the minimal weight~$1$ is concerned).

With the $\ell$-adic and the mod~$\ell$ cases of irreducible odd $2$-dimensional representations of $G_\QQ := \Gal(\Qbar/\QQ)$ essentially settled,
it is natural to wonder what happens in between, i.e.\ modulo prime powers.
Quite some research has been done, but the picture is far from clear. In fact, very basic questions are still open.

\subsubsection*{Modularity modulo prime powers}

Let us consider a representation (continuous like all representations in this paper)
$$ \rho: G_\QQ \to \GL_2(\cO/\lambda^m)$$
with $\cO$ the valuation ring of a finite extension of~$\QQ_\ell$ and $\lambda$ its valuation ideal.
It turns out that the {\em modularity} of $\rho$ follows from known results if one supposes that the residual representation
$\rhobar$ is absolutely irreducible and odd and satisfies certain technical local conditions.
This was surely known to many experts and had been, for instance, discussed on mathoverflow.
We make this precise in Section~\ref{sec:modularity}.

An important point here is that one needs to use the right notion of {\em modularity}. This difficulty is not visible
when only working $\ell$-adically or modulo~$\ell$.
The second author together with Chen and Kiming introduced in \cite{CKW} three notions of modularity modulo prime powers:
{\em strong modularity}, {\em weak modularity}, {\em dc-weak modularity}. These three notions stem from three notions
of Hecke eigenforms modulo prime powers, also called {\em strong}, {\em weak}, and {\em dc-weak}, which we briefly explain now.

Throughout this article, we understand by a {\em Hecke eigenform}~$f$ with coefficients in a ring~$R$
(they are all normalised and almost all cuspidal without this being said explicitly)
a ring homomorphism $f: \TT \to R$, where $\TT$ is a Hecke algebra (to be specified very soon).
We often think of $f$ as the $q$-expansion $\sum_{n \ge 1} f(T_n) q^n \in R[[q]]$.
Let $\TT_k(\Gamma)$ be the full Hecke algebra, generated as a ring by all Hecke operators $T_n$,
acting faithfully on the space of holomorphic cusp forms $S_k(\Gamma)$ of weight~$k$ and level~$\Gamma$.

A {\em weak} Hecke eigenform of weight~$k$ and level~$\Gamma$ with coefficients in~$R$ is a ring homomorphism
$f: \TT_k(\Gamma) \to R$. It is called {\em strong} if there exists an order $\cO$ in a number field together
with a ring homomorphism $\pi: \cO \to R$ such that $f$ factors as $\TT_k(\Gamma) \to \cO \xrightarrow{\pi} R$.
By embedding the order $\cO$ into~$\CC$, the first arrow leads to $\TT_k(\Gamma) \to \CC$, a holomorphic Hecke eigenform.
In simple terms, strong Hecke eigenforms with coefficients in~$R$ are those that are obtained by applying $\pi$ to the
coefficients of a holomorphic eigenform.

Put $S_{\le b}(\Gamma) = \bigoplus_{k=1}^b S_k(\Gamma)$ and let $\TT_{\le b}(\Gamma)$ be the full Hecke algebra acting faithfully
on it. A ring homomorphism $f: \TT_{\le b}(\Gamma) \to R$ is called a {\em dc-weak eigenform} of level~$\Gamma$
(and weights $\le b$; in fact, $b$ will not play any role as long as it is large enough).
A dc-weak eigenform can hence have contributions from many different weights, as is the case for divided congruences,
which is what the abbreviation `dc' stands for.
If $R$ is a finite field or $\Fbar_\ell$, all three notions coincide by the Deligne-Serre lifting lemma (for a presentation
in the setup used here, see \cite[Lemma 16]{CKW}), but they are different in general.
An example that strong is stronger than weak modulo~$\ell^m$ for $m>1$, even if one allows the weight to change
(but not the level), is given in \cite[\S2.5]{KRW}.

As rings of coefficients $R$, we take in this article rings of the form $\cO/\lambda^m$ where $\cO$ is the valuation ring
of a finite field extension of~$\QQ_\ell$, $\lambda$ is its valuation ideal and $m$ a positive integer.
Many results can be and are phrased in this way.
However, a general difficulty exists: we often need to compare two eigenforms, one with coefficients in 
$\cO_1/\lambda_1^{m_1}$, the other one with coefficients in $\cO_2/\lambda_2^{m_2}$.
One then needs to find a ring containing both. In order for such a ring to exist, it is necessary that
$\lambda_i^{m_i} \cap \ZZ_\ell$ for $i=1,2$ both yield the same power of~$\ell$, say $\ell^m$.
This led the second author together with Taix\'es i Ventosa \cite{Taixes2010} to introduce the ring
$$\Zmod{\ell^m} = \Zbar_\ell / \{x \in \Zbar_\ell \;|\; v(x) > m-1\},$$
where $v$ denotes the normalised valuation, i.e.\ $v(\ell)=1$.
We always consider $\Zmod{\ell^m}$ with the discrete topology.
We have $\Zmod{\ell} = \Fbar_\ell$ and for the valuation ring $\cO$ of any finite extension of~$\QQ_\ell$
with absolute ramification index~$e$ and valuation ideal~$\lambda$, the quotient $\cO/\lambda^{e(m-1)+1}$ injects into $\Zmod{\ell^m}$.
This quotient is the smallest one that extends $\ZZ/\ell^m\ZZ$.
The ring $\Zmod{\ell^m}$ is a local $\ZZ/\ell^m\ZZ$-algebra of Krull dimension~$0$ with residue field $\Fbar_\ell$ and
the ring extension $\ZZ/\ell^m\ZZ \subseteq \Zmod{\ell^m}$ is integral.
Any finitely generated subring $R$ of $\Zmod{\ell^m}$ is contained in some ring $\cO/\lambda^{e(m-1)+1}$ as above.
These are free as $\ZZ/\ell^m\ZZ$-modules, but this is not true for all finite subrings of $\Zmod{\ell^m}$.
A Hecke eigenform with coefficients in $\Zmod{\ell^m}$ shall simply be called a {\em modulo~$\ell^m$ Hecke eigenform}.
Dc-weak (and hence also weak) Hecke eigenforms modulo~$\ell^m$ have attached Galois representations,
under the condition that the residual representation is absolutely irreducible (see \cite[Theorem 3]{CKW}).

\subsubsection*{Weight lowering and finiteness}

Let us recall now that in a fixed prime-to-$\ell$ level, there are only finitely many modular Galois representations
with coefficients in $\Fbar_\ell$, which can all be realised -- up to twist -- in weights up to $\ell+1$ and that there
is an explicit recipe for the minimal weight.

A natural question is whether there is a recipe for a minimal weight for strong eigenforms modulo~$\ell^m$, i.e.\ whether (almost) all
the (prime-indexed) Hecke eigenvalues of a given strong eigenform $f$ modulo~$\ell^m$ in weight~$k$ and level~$N$
(prime to~$\ell$) also occur for a strong eigenform $g$ modulo~$\ell^m$ in the same level~$N$ and a `low' or `minimal' weight that can be
calculated from the restriction to a decomposition group at~$\ell$ of the Galois representation attached to~$f$ (under
the assumption of residual absolute irreducibilty).

This question seems to be very difficult. One is then led to consider the question, for fixed prime-to~$\ell$ level~$N$, whether the set
$$ \{ \sum_{n \ge 1} f(T_n) q^n \in \Zmod{\ell^m}[[q]] \;|\; f \textnormal{ strong eigenform modulo $\ell^m$ of level~$N$, any weight }\}$$
is finite. It can also be seen as the set of reductions modulo~$\ell^m$ of all holomorphic Hecke eigenforms in level~$N$ of any weight.
The second author together with Kiming and Rustom conjectures that this is the case (\cite[Conjecture 1]{KRW}).
As is shown in Theorem~2 of loc.~cit., a positive answer to a question of Buzzard \cite[Question 4.4]{Buzzard} would indeed imply this.
As an indication towards finiteness or the potential existence of a weight recipe as alluded to above,
\cite[Theorem 3]{KRW}, proved with the help of Frank Calegari, shows that for $\ell \ge 5$,
there exists a bound $B=B(N,\ell^m)$ such that the $q$-expansion
of any {\em strong} Hecke eigenform modulo~$\ell^m$ of level~$N$, but any weight, already occurs in weight~$k \le B$
for some {\em weak} Hecke eigenform modulo~$\ell^m$ of level~$N$.
One should compare this with the level raising and level lowering results below, which also `only' lead to weak forms.

Some first experimentation has led Kiming, Rustom and the second author to state the formula
$$B(N,\ell^m) = 2 \ell^m + \ell^2 + 1$$
for $m \ge 2$. It is consistent with the available computational data, but should not be understood as a conjecture at this point.

\subsubsection*{Level raising}

Led by classical level raising results, one can hope that similar statements are true modulo~$\ell^m$.
It seems that part of the theory indeed carries over from modulo~$\ell$ to modulo~$\ell^m$ eigenforms.
In Section~\ref{sec:raising} we prove a level raising result for weight~$2$ eigenforms on $\Gamma_0(N)$.
For $\ell>2$, this result is as general as possible. Only for $\ell=2$ some rare cases could not be proved.

Let $f: \TT \to \Zmod{\ell^m}$ be a weak modulo~$\ell^m$ Hecke eigenform with Galois representation~$\rho$.
The main idea is to extend Ribet's `classical' geometric approach of level raising to our more general situation.
For that we need to realise~$\rho$ on the Jacobian $J$ of the appropriate modular curve.
It is well known that the Hecke algebra acts faithfully on~$J$. However, we need that $\TT/\ker(f)$, i.e. the image of the weak
eigenform, also acts faithfully on a subgroup of the Jacobian. The natural place is $J(\Qbar)[\ker(f)]$.
It turns out that this faithfulness does not seem to be that clear.
In fact, we currently make use of the `multiplicity one' property for the residual Galois representation on the Jacobian and, equivalently,
the Gorenstein property of the residual Hecke algebra.
Once this faithfulness is established, the proof proceeds by comparing the new and the old subvarieties in level~$Np$
as in Ribet's original work~\cite{Ribet90a}.
One should expect similar limitations when extending level raising modulo~$\ell^m$ to higher weights,
e.g.\ the weight will likely have to be less than~$\ell$ if one wants complete results in order to remain in the multiplicity one situation,
where faithfulness is known and easily obtained from existing results.

\subsubsection*{Level lowering and other results}

Another natural domain is that of level lowering modulo~$\ell^m$.
The principal idea is that one should always be able to find an eigenform giving rise to a given Galois representation
when the level is equal to the Artin conductor of the representation.
An immediate difficulty is then, of course, to define an Artin conductor for Galois representations modulo~$\ell^m$.
It does not seem to be immediately clear how to do this because not every module over $\ZZ/\ell^m\ZZ$ is free, so that there is
no natural analog for the dimension (of, say, inertia invariants) used in the classical Artin conductor.
Nevertheless, one can at least ask whether one can always find a modulo~$\ell^m$ Hecke eigenform of a level which is only
divisible by primes ramifying in the representation.
There are, indeed, two such results, one is due to Dummigan, and the other one due to Camporino and Pacetti. We quote
both in Section~\ref{sec:lowering}. Dummigan's result, similar to our level raising theorem, works geometrically on cohomology, whereas
Camporino and Pacetti use the deformation theory of Galois representations. Both approaches currently seem to lead to
some restrictions (a congruence condition for Dummigan, and unramified coefficients for Camporino--Pacetti).

Concerning generalisations along the lines of level lowering results modulo~$\ell$, which are based on the use of Shimura curves,
the mod~$\ell^m$ Galois representation must first be realised in the cohomology (or the Jacobian) of the appropriate Shimura curve.
This is likely going to lead into faithfulness problems analogous to the one we solved in the level raising result by appealing to
the Gorenstein or multiplicity one condition.

As a further instance of level lowering (though of a slightly different nature), we mention the following result from
\cite[Theorem 5]{CKW}: Any dc-weak eigenform modulo~$\ell^m$ in level~$N\ell^r$ already arises from a dc-weak eigenform modulo~$\ell^m$
in level~$N$, under the hypotheses $\ell \ge 5$ and that the mod~$\ell$ reduction has an absolutely irreducible Galois representation.
It is also shown that even if one starts with a strong eigenform modulo~$\ell^m$, the one in level~$N$ will only be dc-weak, in general.

Another natural direction is to extend companionship results from eigenforms modulo~$\ell$ to $\ell^m$.
This has been successfully performed by Adibhatla and Manoharmayum in \cite{Raja} for odd~$\ell$ and ordinary modular
forms with coefficients unramified at~$\ell$, under certain conditions. In fact, that work is set in the more general world
of Hilbert modular forms. Another companionship result modulo prime powers has been achieved by the first author together
with Adibhatla~\cite{RaPa}.

\subsubsection*{Computations, algorithm and database}

Next to the theoretical and structural motivation for studying modular forms and modularity questions modulo prime powers,
there is also a strong computational driving force: realising $\ell$-adic modular forms on a computer is only possible
up to a certain precision, i.e.\ one necessarily realises modular forms modulo~$\ell^m$.

This also naturally leads to the questions studied in this article. For instance, if one wants to compute modulo which power of~$\ell$
a modular $\ell$-adic Galois representation~$\rho$ of conductor~$Np$ (with $p$ a prime not dividing $N$) becomes unramified at~$p$,
one can test whether the system of Hecke eigenvalues modulo~$\ell^m$ also occurs in level~$N/p$ for $m=1,2,\dots$ until this fails.
If it first fails at~$m+1$, then $\rho$ modulo~$\ell^m$ is known to be unramified at~$p$.
In cases where level lowering modulo~$\ell^m$ is entirely proved, one also gets that $\rho$ modulo~$\ell^{m+1}$ does ramify at~$p$.
The authors know of no other way of obtaining such information of an $\ell$-adic modular Galois representation.

The authors have developed several algorithmic tools for handling modular forms modulo~$\ell^m$ and they have set up a database.
Section~\ref{sec:alg} contains a brief exposition of how to compute decompositions of commutative algebras into local factors in situations
arising from Hecke algebras, and how to perform weak modularity tests explicitly.
Finally, in Section~\ref{sec:database} we describe features of the database of modular form orbits and higher congruences
that we have developed.

\subsubsection*{Acknowledgements}

The authors would like to thank Rajender Adibhatla, Sara Arias-de-Reyna, Gebhard B\"ockle, Frank Calegari, Imin Chen, Shaunak Deo,
Frazer Jarvis, Ian Kiming, Ariel Pacetti, Nadim Ruston
and many others for various discussions about topics on modular Galois representations modulo prime powers.
They also thank Ken Ribet for having pointed out an inaccuracy in a previous version.
Thanks are also due to the referee for a careful reading and useful suggestions.
The second author thanks Gabi Nebe for having explained the simple algorithmic idempotent lifting (Eq.~\eqref{eq:idemlift}) to him a long time ago.

This project was supported by the Luxembourg Research Fund (Fonds National de la Re\-cherche Luxembourg) INTER/DFG/12/10/COMFGREP in the framework
of the priority program 1489 of the Deutsche Forschungsgemeinschaft.

\section{Modularity}\label{sec:modularity}

In this section we prove the following modularity theorem. This theorem has been known to the experts and is a pretty
straight forward application of `bigR=bigT' theorems.

\begin{theorem}\label{thm:modularity}
Let $\ell \ge 5$ be a prime number, let $\Sigma$ be a finite set of primes not containing~$\ell$ and let
$G_{\QQ,\Sigma\cup\{\infty,\ell\}}$ be the Galois group of the maximal extension of~$\QQ$ unramified outside $\Sigma\cup\{\ell,\infty\}$.
Consider a continuous Galois representation
$$ \rho: G_{\QQ,\Sigma\cup\{\infty,\ell\}} \to \GL_2(\Zmod{\ell^m})$$
such that the residual representation $\rhobar$ satisfies:
\begin{itemize}
\item $\rhobar$ is odd,
\item $\rhobar|_{G_{\QQ(\zeta_p)}}$ is absolutely irreducible,
\item $\rhobar|_{G_{\QQ_\ell}} \not\sim \chi \otimes \fmat 1*01$ and $\rhobar|_{G_{\QQ_\ell}} \not\sim \chi \otimes \fmat 1*0{\overline{\epsilon}}$,
for any $\Fbar_\ell$-valued character~$\chi$ of $G_{\QQ_\ell}$ and the mod~$\ell$ cyclotomic character~$\overline{\epsilon}$
(where $*$ may or may not be zero).
\end{itemize}
Let $N$ be the maximal positive integer divisible only by primes in~$\Sigma$ such that there is a newform of level~$N$ (and some weight)
giving rise to~$\rhobar$.

Then $\rho$ is dc-weakly modular of level~$N$, i.e.\ $\rho \cong \rho_f$ with $f$ a dc-weak Hecke eigenform modulo~$\ell^m$ of level~$N$.
\end{theorem}

In the exposition of the theory, we essentially follow Deo's paper~\cite{Deo}. Let us assume the notation and the set-up from
Theorem~\ref{thm:modularity}.
Let $\cO$ be the valuation ring of a finite extension of~$\QQ_\ell$ with ramification index~$e$, valuation ideal~$\lambda$ and
residue field~$\FF$ such that (possibly after conjugation) $\rho$ takes values in $\GL_2(\cO/\lambda^w) \subset \GL_2(\Zmod{\ell^m})$ with $w = e (m-1)+1$.
Let $\TT'_\cO(\Gamma_1(N))$ be defined as the projective limit over~$b$ of $\cO \otimes \TT'_{\le b}(\Gamma_1(N))$ which are defined
precisely like $\TT_{\le b}(\Gamma_1(N))$, but only take Hecke operators $T_n$ with $n$ coprime to $N\ell$ into account.
Similarly, like Deo we define the partially full Hecke algebra $\TT_\cO^\pf(\Gamma_1(N))$
as the projective limit of $\cO \otimes \TT^\pf_{\le b}(\Gamma_1(N))$ by using in addition the operators $U_q$ for primes $q \mid N$.
If we localise at the system of eigenvalues afforded by~$\rhobar$, we denote this by $\rhobar$ in the index.
Accordingly, denote by $R_\rhobar$ the universal deformation ring of~$\rhobar$ for the group $G_{\QQ,\Sigma\cup\{\infty,\ell\}}$
in the category of local profinite $\cO$-algebras with residue field~$\FF$.

\begin{theorem}[B\"ockle, Diamond--Flach--Guo, Gouv\^ea--Mazur, Kisin]\label{thm:bigRbigT}
Assume the set-up of Theorem~\ref{thm:modularity}.
Then $R_\rhobar \cong \TT'_\cO(\Gamma_1(N))_\rhobar$.
\end{theorem}

This is Theorem~5 from~\cite{Deo}. Note that Deo works with pseudo-re\-pre\-sen\-tations, but this comes down to the same thing here
because we assume $\rhobar$ to be irreducible. In the proof, Deo essentially explains why the results of \cite{DiamondFlachGuo} allow
to strengthen the conclusions of \cite{Boeckle}. A similar discussion can also be found in \cite[\S7.3]{EmertonLG}, where the theorem
is, however, not stated in the form we need here.
Alternatively, one can also invoke \cite[Theorem~1.2.3]{EmertonLG} to an $\ell$-adic lift of~$\rho$, provided
such a lift exists. Recent work by Khare and Ramakrishna~\cite{KR} provides a construction in the ordinary case.

We now apply Theorem~\ref{thm:bigRbigT}. By assumption, $\rho$ is a deformation of~$\rhobar$ with the right ramification set,
whence the universality leads to an $\cO$-algebra homomorphism $R_\rhobar \to \cO/\lambda^w$, which we consider as an $\cO$-algebra homomorphism
$$ f :  \TT'_\cO(\Gamma_1(N))_\rhobar \to \cO/\lambda^w.$$
By construction, the Galois representation associated with~$f$ is isomorphic to~$\rho$.

In order to finish the proof of Theorem~\ref{thm:modularity}, $f$ has to be extended to the full Hecke algebra in order to make
it a genuine Hecke eigenform modulo~$\ell^m$.
Next we use that $\TT_\cO^\pf(\Gamma_1(N))_\rhobar$ is finite over $\TT'_\cO(\Gamma_1(N))_\rhobar$. This is proved in \cite[Proposition~6]{Deo};
one should note that the $\Gamma_1(N)$-new assumption is not necessary for this statement (see the proof of \cite[Theorem~3]{Deo}).
This integrality allows us to extend $f$ to an $\cO$-algebra homomorphism
$$ f :  \TT^\pf_\cO(\Gamma_1(N))_\rhobar \to \tilde{\cO}/\tilde{\lambda}^{\tilde{w}}$$
where $\cO \subseteq \tilde{\cO}$ is the valuation ring of some finite extension of~$\QQ_\ell$ with valuation ideal $\tilde{\lambda}$
and ramification index~$\tilde{e}$ and $\tilde{w} = \tilde{e}(m-1)+1$. One is able to make this extension because one only needs to find one zero
in some ring of the form $\tilde{\cO}/\tilde{\lambda}^{\tilde{w}}$ for any monic polynomial with coefficients in $\cO/\lambda^w$;
that this is possible follows, for instance, by choosing any monic lift to~$\cO$.
From the natural degeneracy map, we next get an $\cO$-algebra homomorphism
$ f :  \TT^\pf_\cO(\Gamma_1(N\ell))_\rhobar \to \tilde{\cO}/\tilde{\lambda}^{\tilde{w}}$,
which after choice of $f(U_\ell)$ leads to the $\cO$-algebra homomorphism
$$ f :  \TT^\pf_\cO(\Gamma_1(N\ell))_\rhobar[[U_\ell]] \to \tilde{\cO}/\tilde{\lambda}^{\tilde{w}}.$$
According to \cite[Proposition~5]{Deo}, one can identify $\TT^\pf_\cO(\Gamma_1(N\ell))_\rhobar[[U_\ell]]$ with a quotient of the full
Hecke algebra $\TT_\cO(\Gamma_1(N\ell))$. We obtain thus an $\cO$-algebra homomorphism
$$ f : \TT_\cO(\Gamma_1(N\ell)) \to \tilde{\cO}/\tilde{\lambda}^{\tilde{w}}.$$
As its image is finite, it will factor through $\cO \otimes \TT_{\le b}(\Gamma_1(N\ell))$ for a suitable weight bound~$b$,
so that we finally get a ring homomorphism
$$ f:  \TT_{\le b}(\Gamma_1(N\ell)) \to \tilde{\cO}/\tilde{\lambda}^{\tilde{w}}.$$
This is the dc-weak eigenform that is needed to finish the proof of Theorem~\ref{thm:modularity}.
Note that one can still remove $\ell$ from the level of the final form because of \cite[Theorem 5]{CKW}.

\section{Level raising via modular curves}\label{sec:raising}

Let $p$ be a rational prime. Then one has a natural inclusion map
$$S_k(\Gamma_0(N))\oplus S_k(\Gamma_0(N)) \to S_k(\Gamma_0(Np)),$$
the image of which is called the {\em $p$-old subspace}.
This subspace is stable under the action of $\TT_k(Np):= \TT_k(\Gamma_0(Np))$ and so is its orthogonal complement under the Petersson inner product.
This complementary subspace is called the {\em $p$-new subspace} and we denote by $\TT_k^{p-\textrm{new}}(Np)$ the quotient of $\TT_k(Np)$ that acts faithfully on it.
We will call this quotient the {\em $p$-new quotient of $\TT_k(Np)$}. There is also the {\em $p$-old quotient} that is defined in the obvious way.

We can now state the main level raising result of this article.
\begin{theorem}\label{thm:main}
Let $R$ be a local topological ring with maximal ideal $\fm_R$.
Let $\rho : G_\QQ \to \GL_2(R)$ be a continuous Galois representation that is modular, associated with a weak eigenform $\theta:\TT_2(N)\to R$,
and such that the residual representation $\rhobar: G_\QQ \to \GL_2(R/\fm)$ is absolutely irreducible.
If the characteristic of $R/\fm$ is~$2$, assume the {\em multiplicity one/Gorenstein condition} that $\rhobar$ is not unramified at~$2$ with scalar Frobenius.

Let $p$ be a prime which satisfies the {\em level raising condition} for~$\rho$ by which we mean that $\rho$ is unramified at~$p$ and
$$\Tr(\rho(\Frob_p)) = \pm (p+1).$$
Then the image of $\theta$ is a finite ring, $R/\fm$ is a finite field and
$\rho$ is also associated with a weak eigenform $\theta':\TT_2(Np)\to R$ which is new at $p$,
i.e.\ $\theta'$ factors through $\TT_2^{p-\textrm{new}}(Np)$.
\end{theorem}

In view of Lemma~\ref{lem:aux}, the following corollary is essentially just an equivalent reformulation.
Let $\cO$ be the ring of integers of a number field and $\la$ a prime in $\cO$ above $\ell$.

\begin{corollary}\label{cor:main}
Let $m\geq 1$ be an integer and $\rho : G_\QQ \to \GL_2(\cO/\lambda^m)$ be a continuous
(for the discrete topology on $\cO/\lambda^m$) Galois representation
that is modular, associated with a weak eigenform $\theta:\TT_2(N)\to\cO/\lambda^m$,
and such that the residual representation $\rhobar: G_\QQ \to \GL_2(\cO/\la)$
is absolutely irreducible.
If $\cO/\la$ is of characteristic~$2$, assume the {\em multiplicity one/Gorenstein condition} that $\rhobar$ is not unramified at~$2$ with scalar Frobenius.

Let $p$ be a prime which satisfies the {\em level raising condition} for~$\rho$, which means here that 
$$(\ell N,p)=1 \textnormal{ and }\Tr(\rho(\Frob_p)) \equiv \pm (p+1) \mod \la^m.$$
Then $\rho$ is also associated with a weak eigenform $\theta':\TT_2(Np)\to\cO/\lambda^m$ which is new at $p$,
i.e.~$\theta'$ factors through the $\TT_2^{p-\textrm{new}}(Np)$.
\end{corollary}

We remark that for $m=1$ this is Theorem 1 of \cite{Ribet90a}.
Even if $\theta$ is a strong eigenform,
there is no guarantee that the weak eigenform of level new at $p$ that one obtains in the end is strong.

\begin{corollary}
Let $R$ be a local topological ring with maximal ideal $\fm_R$ and
let $\rho : G_\QQ \to \GL_2(R)$ be a continuous Galois representation that is modular, has finite image
and such that the residual representation $\rhobar: G_\QQ \to \GL_2(R/\fm)$ is absolutely irreducible.
If the characteristic of $R/\fm$ is~$2$, assume the {\em multiplicity one/Gorenstein condition} that $\rhobar$ is not unramified at~$2$ with scalar Frobenius.

Then there exists a positive set of primes $p$ (coprime to $N$) such that $\rho$ is modular of level $Np$ and new at~$p$.
\end{corollary}

\begin{proof}
This is proved as in~\cite{Ribet90a}.
The argument is that complex conjugation, as an involution, has trace~$0$ and determinant~$-1$.
By Chebotarev's density theorem, there is a positive density set of primes $p$ such that $-1=\det(\rho(\Frob_p)) = p $
and $p+1 = 0 = \Tr(\rho(\Frob_p))$ in~$R$.
\qed\end{proof}

\subsubsection*{Jacobians of modular curves}

In what follows we set $\TT_N := \TT_2(\Gamma_0(N))$ and $\TT_{Np} := \TT_2(\Gamma_0(Np))$.
The approach taken here is adapted from Ribet's original one, i.e.\ it is based on the geometry of modular curves and their Jacobians.
In this section we gather the necessary results from \cite{Ribet90a} that we need for the proof of the main result.
Let $N$ be a positive integer. Let $X_0(N)$ be the modular curve of level $N$ and $J_0(N):=\operatorname{Pic}^0(X_0(N))$ its Jacobian.
There is a well defined action of  the Hecke operators $T_n$ on $X_0(N)$ and hence, by functoriality, on $J_0(N)$, too.
The dual of $J_0(N)$ carries an action of the Hecke algebra as well and can be identified with $S_2(\Gamma_0(N))$.
This implies that one has a faithful action of $\TT_N$ on $J_0(N)$. 

Let now $p$ be a prime not dividing~$N$.
In the same way one has an action of Hecke operators on $X_0(Np)$ and its Jacobian $J_0(Np)$ and the latter admits a faithful action of $\TT_{Np}$.
The moduli interpretation of $X_0(N)$ and $X_0(Np)$ allows us to define the two natural degeneracy maps
$\delta_1, \delta_p:X_0(Np)\to X_0(N)$
and their pullbacks $\delta_1^*, \delta_p^*:J_0(N)\to J_0(Np)$.
The image of the map
\begin{equation*}
\alpha:J_0(N)\times J_0(N)\to J_0(Np),\textrm{\ \ \ \ \ \ \ }(x,y)\mapsto \delta_1^*(x)+\delta_p^*(y).
\end{equation*}
is by definition the {\em $p$-old subvariety} of $J_0(Np)$.
We will denote it by~$A$. The map $\alpha$ is \emph{almost} Hecke-equivariant:
\begin{equation}
\alpha\circ T_q = T_q\circ\alpha\textrm{ for every prime }q\neq p , \label{eq:1}
\end{equation}
\begin{equation}
\alpha\circ\fmat{T_p}{p}{-1}{0} = U_p\circ\alpha. \label{eq:2}
\end{equation}
For the first equation to make sense one interprets the operator $T_q$ on the left hand side of equation~\eqref{eq:1}
as acting diagonally on $J_0(N)\times J_0(N)$.
We also work under the notational convention $T_q = U_q$ for primes $q \mid N$, but we write $U_p$ in level $Np$.
Consider also the kernel $\Sh$ of the map $J_0(N) \to J_1(N)$ induced by $X_1(N)\to X_0(N)$.
If we inject it into $J_0(N)\times J_0(N)$ via $x \mapsto (x, -x)$ then its image, which we will denote by $\Sigma$,
is the kernel of~$\alpha$ (see Proposition~1 in \cite{Ribet90a}).

Let $\Delta$ be the kernel of $\fmat{1+p}{T_p}{T_p}{1+p}\in M^{2\times2}(\TT_N)$ acting on $J_0(N)\times J_0(N)$.
The group $\Delta$ is finite and comes equipped with a perfect $\mathbb{G}_m$-valued skew-symmetric pairing.
Furthermore $\Sigma$ is a subgroup of $\Delta$, self orthogonal, and $\Sigma \subseteq \Sigma^\perp \subseteq \Delta$.
One can also see $\Delta/\Sigma$, and therefore its subgroup $\Sigma^\perp/\Sigma$, as a subgroup of~$A$.

Let $B$ be the $p$-new subvariety of $J_0(Np)$. It is a complement of $A$, i.e.\ $A+B = J_0(Np)$ and $A\cap B$ is finite.
The Hecke algebra acts faithfully on~$B$ through its $p$-new quotient and it turns out (see Theorem 2 in \cite{Ribet90a}) that
\begin{equation}\label{iso:1}
A\cap B \isom \Sigma^\perp/\Sigma.
\end{equation}
as groups.

Furthermore $\Sh$, and therefore $\Sigma$ and its Cartier dual $\Delta/\Sigma^\perp$, are annihilated by the operators $\eta_r = T_r - (r+1)\in \TT_N$
for all primes $r\nmid Np$ (see Proposition 2 in \cite{Ribet90a}).
In this context, we recall that a maximal ideal $\fm$ of the Hecke algebra $\TT_N$ is called {\em Eisenstein} if
$T_r \mod \fm$ equals the Frobenius traces of a two-dimensional reducible Galois representation at almost all primes~$r$.
This is in particular the case if $\fm$ contains the operator $T_r - (r+1)$ for almost all primes~$r$.
Consequently, any maximal ideal in the support of the Hecke modules $\Sigma$ and $\Delta/\Sigma^\perp$ is Eisenstein.

\subsubsection*{Proof of Theorem \ref{thm:main}}

We assume the setting of Theorem~\ref{thm:main}.
In particular, we assume that $\rho$ satisfies the level raising condition at a prime~$p \nmid N$, i.e.\
there is $\epsilon \in \{\pm 1\}$ such that $\theta(T_p)=\Tr(\rho(\Frob_p)) = \epsilon(p+1)$.
Let $\bar{\theta}:\TT_N\to R/\fm_R$ be its reduction modulo $\fm_R$
(which is associated with $\rhobar$, the modulo $\fm_R$ reduction of $\rho$),
and let $I$ and $\fm$ be the kernels of $\theta$ and $\bar{\theta}$, respectively.
It will be enough to find a weak eigenform $\theta':\TT_{Np}\to R$ (i.e.\ a ring homomorphism) that agrees with $\theta$ on $T_q$
for all primes $q\neq p$ and factors through $\TT_{Np}^{p-\textrm{new}}$ (hence, new at $p$).

\begin{lemma}\label{lem:aux}
The ideal $\fm$ is the only maximal ideal of $\TT_N$ containing~$I$.
Moreover, $\TT_N/I$ is a finite subring of~$R$ of positive characteristic a prime power~$\ell^r$.
\end{lemma}

\begin{proof}
Since $\TT_N$ is a $\ZZ$-Hecke algebra acting faithfully on $S_2(N)$ we have that $\TT_N$ injects into $M^{d\times d}(\ZZ)$,
where $d$ is the dimension of $S_2(N)$. We can therefore see every operator in $\TT_N$  as an integral matrix of dimension~$d$.
We recall that the eigenvalues of the operator $T_n$ will correspond to the coefficients $a_n(f)$
when $f$ runs through the normalised eigenforms in $S_2(N)$.

Let $g(X)\in \ZZ[X]$ be the characteristic polynomial of $T_p$. The hypothesis $\theta(T_p)=\epsilon(p+1)$ implies that $T_p-\epsilon(p+1)\in I$
and therefore $m:=g(\epsilon(p+1))\in I$.
Since $p\nmid N$, the Ramanujan-Petersson bounds guarantee that none of the eigenvalues of $T_p$ is equal to $\epsilon(p+1)$ and therefore $m$ is non-zero.
We thus have that $(m)\subseteq I$. This makes the quotient $\TT/I$ finite.

Since $\TT/I$ is Artinian, it can be written as a direct product of Artinian local rings indexed by its finitely many maximal ideals.
Assume it decomposes as a direct product of $s$ local rings, with $s\geq 1$.
The set containing the identity $e_i$ of each component then forms a complete set (i.e.\ $\sum_{i=1}^se_i=1$) of pairwise orthogonal
(i.e.\ $e_ie_j=0$ for $1\leq i\neq j\leq s$) non-trivial (i.e.\ $e_i\neq0,1$) idempotents for $\TT_N/I$.
The set $\{\bar{e}_1,\ldots, \bar{e}_s\}$ of their image through the injection of $\TT_N/I$ into $R$ is clearly
a complete set of pairwise orthogonal non-trivial idempotents, too.
This implies that $R$ is isomorphic to $\prod_{i=1}^s\bar{e}_iR$.
But this cannot happen unless $s=1$ since $R$ is local.
Since $s=1$ we get that $\TT_N/I$ is local as well. The claims are then immediate.
\qed\end{proof}

By the previous lemma, we have inclusions $(\ell^r) \subseteq I \subseteq \fm$ with some prime power $\ell^r > 1$, giving rise to inclusions
$$ V[\ell^r] := J_0(N)(\Qbar)[\ell^r] \supseteq V[I] := J_0(N)(\Qbar)[I]\supseteq V[\fm] := J_0(N)(\Qbar)[\fm].$$

\begin{lemma}\label{lem:VI}
The support of $V[I]$ is the singleton~$\fm$ and is hence non-Eisenstein.
\end{lemma}

\begin{proof}
As $V[I] \supseteq V[\fm]$, the maximal ideal $\fm$ is in the support of $V[I]$.
Since the representation $\rhobar$ is irreducible we get that $\fm$ is non-Eisenstein
(see for example Theorem 5.2c in~\cite{Ribet90b}).
Finally, Lemma \ref{lem:aux} implies that $\Supp(V[I])$ is the singleton~$\{\fm\}$.
\qed\end{proof}

\begin{lemma}\label{lem:VI-alpha}
The restriction of~$\alpha$ to $V[I]$ is injective and its image $\alpha(V[I])$ is stable under the action of~$\TT_{Np}$.
In particular, $U_p$ acts on $\alpha(V[I])$ by multiplication by~$\epsilon$.
\end{lemma}

\begin{proof}
Consider the image of $V[I]$ (still denoted $V[I]$) under the $\TT_N$-equivariant embedding
$$J_0(N)\xrightarrow{x \mapsto (x,-\epsilon x)} J_0(N)\times J_0(N).$$
Next recall that the kernel $\Sigma$ of $J_0(N)\times J_0(N) \xrightarrow{\al} A\subseteq J_0(Np)$
is annihilated by almost all operators $T_r - (r+1)$ with $r$ prime. The fact that the support of $V[I]$ is non-Eisenstein
from Lemma~\ref{lem:VI} shows that the intersection of $\Sigma$ and $V[I]$ is trivial, proving the injectivity of $\alpha|_{V[I]}$.

As $\alpha$ commutes with the action of the Hecke operators $T_n$ with $n$ coprime to~$p$ (see Equation~\eqref{eq:1}),
it follows that $\alpha(V[I])$ is stable under those operators.
Here the level raising condition enters for proving the stability under $U_p$, as follows by using Equation~\eqref{eq:2} for
$y \in V[I]$:
\begin{align*}
U_p(y) & = U_p(\al(x,-\epsilon x))
         = \al(\fmat{T_p}{p}{-1}{0} \fvect{x}{-\epsilon x})
         = \al(T_p(x) - \epsilon px, -x)\\
       & = \al(\epsilon (p+1) x - \epsilon px, -x) = \al(\epsilon x,-x) = \epsilon \al(x,-\epsilon x) = \epsilon y.
\end{align*}
The final claim follows as well.
\qed\end{proof}

The following proposition is a non-trivial input.

\begin{proposition}\label{prop:VI-ff}
The $\TT_N/I$-module $V[I]$ is faithful.
\end{proposition}

\begin{proof}
Due to the assumptions, Theorem~9.2 of~\cite{Edixhoven1992} implies that $V[\fm]$ is of dimension~$2$ as $\TT_N/\fm$-module.
By Nakayama's Lemma, it follows that the localisation at~$\fm$ of the $\ell$-adic Tate module is free of rank~$2$ as $(\TT_N \otimes_\ZZ \ZZ_\ell)_\fm$-module and that
$\Hom_{\ZZ_\ell}((\TT_N \otimes_\ZZ \ZZ_\ell)_\fm,\ZZ_\ell)$ is free of rank~$1$ as $(\TT_N \otimes_\ZZ \ZZ_\ell)_\fm$-module, precisely
as on p.~333 of~\cite{Tilouine}. Consequently,
$$V[\ell^r]_\fm \cong (\TT_N/\ell^r \TT_N)_\fm^2 \cong \Hom_{\ZZ_\ell}((\TT_N/\ell^r \TT_N)_\fm,\ZZ/\ell^r\ZZ)^2,$$
which implies by taking the $\overline{I}$-kernel with $\overline{I}$ the ideal such that $\TT_N/I \cong (\TT_N/\ell^r\TT_N)/\overline{I}$ that
$$V[I] \cong \Hom_\ZZ((\TT_N/\ell^r\TT_N)/\overline{I}, \ZZ/\ell^r\ZZ)^2 = \Hom_\ZZ(\TT_N/I, \ZZ/\ell^r\ZZ),$$
showing that $V[I]$ is faithful as $\TT_N/I$-module.
\qed\end{proof}

The authors do not know if the `multiplicity one' or `Gorenstein' condition is necessary. In the remaining case, the $2^r$-torsion group scheme
is ordinary, and hence by arguments as in Corollary~2.3 of \cite{multone} admits a nice decomposition as
$$ 0 \to (\TT_N/\ell^r \TT_N)_\fm \to V[\ell^r]_\fm \to \Hom_\ZZ((\TT_N/\ell^r\TT_N)_\fm, \ZZ/\ell^r\ZZ) \to 0. $$
However, we do not know if this sequence remains exact after taking the $\overline{I}$-kernel.
If this were the case, the additional assumption would be unnecessary.

\begin{lemma}
The action of $\TT_{Np}$ on $\alpha(V[I])$ is given by a ring homomorphism $\theta':\TT_{Np}\to R$ satisfying
$\theta'(T_q) = \theta(T_q)$ for all primes $q \neq p$ and $\theta'(U_p) = \epsilon$.
In particular, $\theta$ and $\theta'$ give rise to isomorphic Galois representations.
\end{lemma}

\begin{proof}
The faithfulness of $V[I]$ as $\TT/I$-module from Proposition~\ref{prop:VI-ff} implies that $\theta$ factors through a subring~$S$ of $\End(V[I])$,
which is also a subring of~$R$.
By Lemma~\ref{lem:VI-alpha} and Equation~\eqref{eq:1},
the action of $\TT_{Np}$ on $\alpha(V[I])$ is also given by elements of~$S$, leading to a ring homomorphism $\theta':\TT_{Np} \to S \subseteq R$.
\qed\end{proof}

To finish the proof of Theorem~\ref{thm:main},
it remains to show that $\theta'$ factors through the $p$-new quotient of~$\TT_{Np}$.
To this end, it is enough to show that $\alpha(V[I])$ is a subgroup of $A\cap B$.
We again proceed according to Ribet.
By the level raising condition, $V[I]$, when considered as a subgroup of $J_0(N)\times J_0(N)$, is a subgroup of $\Delta$,
whence $\alpha(V[I]) \subseteq \Delta/\Sigma$.
As $\Delta/\Sigma^\perp$ is Eisenstein but $\alpha(V[I])$ is not, $\alpha(V[I])/\Sigma^\perp = 0$.
This implies $\alpha(V[I]) \subseteq \Sigma^\perp/\Sigma = A \cap B$, completing the proof of Theorem~\ref{thm:main}.

\section{Level lowering}\label{sec:lowering}

In this section we give an overview of results about level lowering modulo prime powers.
We start by the following simple observation: twisting an eigenform $f$ by a Dirichlet character~$\chi$ such that $\chi \equiv 1 \mod \lambda^m$
leads to an eigenform $g = f \otimes \chi$, which is congruent to~$f$ modulo~$\lambda^m$.
This idea leads to the following two level lowering results from the first author's unpublished PhD thesis~\cite{panosPhD}.

\begin{proposition}[Split ramified case]
Let $f \in S_k(\Gamma_1(M))$ be a newform such that the restriction to a decomposition group at~$p \neq \ell$
of the $\ell$-adic Galois representation attached to~$f$ is isomorphic to $\chi_1 \oplus \chi_2$, where both characters ramify.
Let $\lambda$ be a prime ideal of a number field containing the coefficients of~$f$.

If $\chi_1$ is unramified modulo~$\lambda^m$, then there exists a normalised eigenform $g \in S_k(\Gamma_1(M/p))$ such that $f \equiv g \mod \lambda^m$.
\end{proposition}

\begin{proof}
We can decompose $\chi_1 = \chi_{1,\unram} \chi_{1,\ram}$ into an unramified and a ramified character of $G_{\QQ_p}$.
As $p \neq \ell$, the order of $\chi_{1,\ram}$ is finite.
By assumption,  $\chi_{1,\ram} \equiv 1 \mod \lambda^m$, whence in particular the order of $\chi_{1,\ram}$ is a power of~$\ell$ because
only roots of unity of $\ell$-power order vanish under reduction modulo~$\lambda$. Thus $\chi_{1,\ram}$ is tamely ramified.
By the local and the global Kronecker-Weber theorems, $\chi_{1,\ram}$ can be seen as a global Dirichlet character $\tilde{\chi}_{1,\ram}$ of conductor~$p$
the restriction of which to $G_{\QQ_p}$ equals $\chi_{1,\ram}$.

Let now $g$ be the newform corresponding to the twist $f \otimes \tilde{\chi}_{1,\ram}^{-1}$.
Then the restriction to a decomposition group at~$p$ of the $\ell$-adic Galois representation attached to~$g$ is isomorphic
to $\chi_{1,\unram} \oplus \chi_2 \chi_{1,\ram}^{-1}$.
If $\chi_2$ is tame (i.e.\ of conductor~$p$), then $\chi_2 \chi_{1,\ram}^{-1}$ is either tame or unramified, and in any case its conductor divides~$p$.
If $\chi_2$ is wild, i.e.\ it factors through $\Gal(\QQ_p(\zeta_{p^rN})/\QQ_p)$ with $r \ge 2$ and $p \nmid N$, but not through
$\Gal(\QQ_p(\zeta_{p^{r-1}N})/\QQ_p)$, then also $\chi_2 \chi_{1,\ram}^{-1}$ factors through $\Gal(\QQ_p(\zeta_{p^rN})/\QQ_p)$ but not through $\Gal(\QQ_p(\zeta_{p^{r-1}N})/\QQ_p)$,
whence the conductor of $\chi_2 \chi_{1,\ram}^{-1}$ equals that of~$\chi_2$. In both cases we hence find that
the conductor of $\chi_2 \chi_{1,\ram}^{-1}$ divides the conductor of~$\chi_2$.
Since the $p$-valuation of~$M$ equals the $p$-valuation of the conductor of~$\chi_2$ plus~$1$ (since the conductor of $\chi_1$ is~$p$)
and the $p$-valuation of the newform level of~$g$ is the $p$-valuation of the conductor of $\chi_2 \chi_{1,\ram}^{-1}$,
it is clear that the newform level of~$g$ divides $M/p$.
\qed\end{proof}

\begin{proposition}[Special ramified case]
Let $f \in S_k(\Gamma_1(M))$ be a newform such that the restriction to a decomposition group at~$p \neq \ell$
of the $\ell$-adic Galois representation attached to~$f$ is isomorphic to $\chi \otimes \fmat \omega * 0 1$, where $\chi$ and $*$ ramify and $\omega$ is
the $\ell$-adic cyclotomic character.
Let $\lambda$ be a prime ideal of a number field containing the coefficients of~$f$.

If $\chi$ is unramified modulo~$\lambda^m$, then there exists a newform $g \in S_k(\Gamma_1(M/p))$ such that $f \equiv g \mod \lambda^m$.
\end{proposition}

\begin{proof}
The proof is essentially the same as in the split ramified case. Note, however, that the tameness of $\chi$ implies that $p^2$ exactly divides~$M$,
whence the newform level of~$g$ will be exactly~$M/p$.
\qed\end{proof}

These propositions may be useful in some situations. 
We also remark that the only Dirichlet character that is trivial modulo~$\ell^2$ in the sense of being equal to $1 \in \Zmod{\ell^2}$
is the trivial one.
That is just due to the fact that $\lambda := 1-\zeta_\ell$ is a uniformiser
of $\QQ_\ell(\zeta_\ell)$, whence $\zeta_\ell \not\equiv 1 \mod (\lambda)^2$.
This implies that the level does not lower modulo $\ell^m$ for any $m \ge 2$ at primes~$p$ satisfying the hypothesis of one of the
preceding propositions.
We now quote the main result from~\cite{Dummigan2015}, including the discussion in the last paragraph of that article.

\begin{theorem}[Dummigan]\label{thm:dummigan}
Let $\ell$ be a prime.
Let $\ell+2 > k \ge 2$ and let $p$ be a prime not dividing $N \in \NN$ such that $p \not \equiv 1 \mod \ell$.
Let $f \in S_k(\Gamma_1(Np))$ be an eigenform and let $\lambda$ be a prime of the coefficient field of~$f$ above~$\ell$.
Suppose that the residual Galois representation of~$f$ modulo~$\lambda$ is irreducible.

If for some $m \ge 1$ the Galois representation of~$f$ modulo~$\lambda^m$ is unramified at~$p$,
then there is a weak eigenform~$g$ of weight~$k$ and level~$\Gamma_1(N)$ such that $f \mod \lambda^m$ equals~$g$
at all coefficients the index of which is coprime to~$p$.
\end{theorem}

Dummigan also gives an explicit example where the resulting form~$g$ cannot be strong.
We include another still unpublished result from~\cite{PacettiCamporino} on level lowering,
which is proved using the deformation theory of Galois representations.

\begin{theorem}[Pacetti-Camporino]
Let $\ell \ge 7$ be a prime.
Let $2 \le k \le \ell-1$.
Let $M$ be a positive integer.
Let $f \in S_k(\Gamma_1(M))$ be an eigenform with coefficients in $K_f$.
Let $\cO_f$ be the ring of integers of~$K_f$.
Assume that
\begin{itemize}
\item $\ell$ is unramified in~$\cO_f$, and
\item $\SL_2(\cO_f/\lambda)$ is a subgroup of the image of the mod~$\lambda$ representation attached to~$f$.
\end{itemize}
If $p \mid M$ is a prime and $m \ge 1$ is an integer such that the modulo~$\lambda^m$ Galois representation
associated with $f$ is unramified at~$p$,
then there is a weak eigenform~$g$ of weight~$k$ and level~$\Gamma_1(M/p)$ such that $f \mod \lambda^m$ equals~$g$
at all coefficients the index of which is coprime to~$p$.
\end{theorem}

This result is proved by first applying techniques of Ramakrishna: by introducing auxiliary primes in order to kill
local obstructions, the authors construct an $\ell$-adic lift in which $p$ remains unramified.
They then prove and use a modularity lifting theorem to obtain that their lift is associated with some newform.
Finally, they apply Theorem~\ref{thm:dummigan} to remove the auxiliary primes, which had been chosen in such a way that
Dummigan's theorem applies.

\section{Computational aspects}\label{sec:alg}

In this section, we describe various algorithms we have implemented and used in our computational study of higher congruences.

\subsubsection*{Some commutative algebra}

We start by summarising some well known facts from commutative algebra.
Let $R$ be an {\em Artinian} ring, i.e.\ a ring in which every descending chain
of ideals becomes stationary.
In particular, for any ideal $\fa$ of~$R$, the sequence $\fa^n$ becomes stationary,
i.e.\ $\fa^n = \fa^{n+1}$ for all $n$ ``big enough''. We will then use the notation $\fa^\infty$ for $\fa^n$.
The following proposition is well known and easy to prove:

\begin{proposition}\label{prop:artin}
Let $R$ be an Artinian ring.
Then every prime ideal of $R$ is maximal and there are only finitely many maximal ideals in~$R$.
Moreover, the maximal ideal~$\fm$ is the only one containing~$\fm^\infty$.
Furthermore, if  $\fm \neq \fn$ are two maximal ideals, then for any $k \in \NN \cup \{\infty\}$, the ideals $\fm^k$ and $\fn^k$
are coprime.
The Jacobson radical $\bigcap_{\fm \in \Spec(R)} \fm$ is equal to the nilradical and consists of the nilpotent elements,
and we have $\bigcap_{\fm \in \Spec(R)} \fm^\infty = (0)$.
Moreover, for every maximal ideal~$\fm$, the ring $R/\fm^\infty$ is local with maximal ideal~$\fm$ and is hence isomorphic to $R_\fm$,
the localisation of $R$ at~$\fm$.
Finally, by virtue of the Chinese Remainder Theorem we have the following isomorphism, referred to as {\em local decomposition}:
$$R \xrightarrow{a \mapsto (\dots, a + \fm^\infty, \dots)} \prod_{\fm \in \Spec(R)} R/\fm^\infty \cong \prod_{\fm \in \Spec(R)} R_\fm. $$
\end{proposition}

\begin{definition}
An {\em idempotent} of a ring~$R$ is an element~$e$ that satisfies $e^2=e$.
Two idempotents $e$, $f$ are {\em orthogonal} if $ef=0$.
An idempotent $e$ is {\em primitive} if it cannot be written as a sum of two idempotents both different from~$0$.
A set of idempotents $\{e_1,\dots, e_n\}$ is said to be {\em complete} if $1 = \sum_{i=1}^n e_i$.
\end{definition}

In concrete terms, a complete set of primitive pairwise orthogonal idempotents is given by
$(1,0,\dots,0), (0,1,0,\dots,0), \dots, (0,\dots,0,1)$.

\begin{proposition}[Newton method/Hensel lifting -- special case]
Let $R$ be a ring and $I$ be an ideal.
Let $f \in R[X]$ be a polynomial.
We assume that there exist $a \in R$ and a polynomial $b\in R[X]$ such that $1 = a f(X) + b(X)f'(X)$.
Let further $a_0 \in R$ be such that $f(a_0) \in I^r$ for some $r \ge 1$.
For $n \ge 1$, we make the following recursive definition:
$$ a_n := a_{n-1} - f(a_{n-1})b(a_{n-1}).$$
Then for all $n \in \NN$, we have $f(a_n) \in (I^r)^{2^n}$.
In particular, if $\bigcap_{n \ge 1} I^n = 0$ then the sequence $f(a_n)$ converges to~$0$ exponentially.
\end{proposition}

\begin{proof}
This is a straight forward calculation with Taylor expansions of the polynomial.
\qed\end{proof}

\begin{corollary}[Algorithmic idempotent lifting]\label{cor:alg-idem}
Let $R$ be a commutative $\ZZ_\ell$-al\-ge\-bra which is finitely generated as $\ZZ_\ell$-module.
Let $e_0 \in R/\ell R$ be an idempotent.
For $n \ge 1$, make the following recursive definition:
\begin{equation}\label{eq:idemlift}
e_n := e_{n-1} - (e_{n-1}^2-e_{n-1})(2e_{n-1}-1) = 3e_{n-1}^2 - 2e_{n-1}^3.
\end{equation}
Then $e_n^2 \equiv e_n \mod \ell^{2^n}R$ for all $n \ge 0$.
Moreover, the $e_n$ form a Cauchy sequence in~$R$ and thus converge to an idempotent $e \in R$
`lifting' $e_0$, i.e.\ the image of~$e$ in~$R/\ell R$ is~$e_0$.
\end{corollary}

\begin{proof}
This is a simple application of the Newton method to the polynomial $f(X) = X^2 - X$.
Note that we have $f'(X) = 2X-1$ and $1 = -4(X^2-X) + (2X-1)(2X-1)$.
\qed\end{proof}

The corollary thus tells us that any idempotent of $R/\ell R$ lifts to an idempotent of~$R$, and it
tells us that the lift can be approximated by a simple recursion formula that is easy to implement and
converges very rapidly.
We shall now apply the preceding considerations to a commutative $\ZZ_\ell$-algebra~$\TT$
which is free and finitely generated as a $\ZZ_\ell$-module.
Let $\Tbar = \TT \otimes \FF_\ell$ and $\TT_{\QQ_\ell} = \TT \otimes \QQ_\ell$.
Note that $\Tbar$ and $\TT_{\QQ_\ell}$ are Artinian rings because they are finite dimensional vector spaces.
The following well-known result follows from the above considerations together with some standard commutative algebra.

\begin{proposition}\label{prop:commalg}
The algebra $\TT$ is equidimensional (in the sense of Krull dimension) of dimension~$1$,
i.e.\ any maximal ideal $\fm$ strictly contains at least one minimal prime ideal~$\lambda$
and there is no prime ideal strictly in between the two.
The maximal ideals of $\TT$ correspond bijectively under taking pre-images to the maximal ideals of~$\Tbar$;
the same letter will be used to denote them.
The minimal primes $\lambda$ of $\TT$ are in bijection with the prime ideals of $\TT_{\QQ_\ell}$ (all of which are maximal)
under extension, for which the notation $\lambda^{(e)}$ will be used.
Under the correspondences, one has $\Tbar_\fm \cong \TT_\fm \otimes \FF_\ell$ and $\TT_\lambda \cong \TT_{\QQ_\ell,\lambda^{(e)}}$.
By virtue of lifts of idempotents and Proposition~\ref{prop:artin}, we have the local decompositions
$$ \TT \cong \prod_\fm \TT_\fm, \Tbar \cong \prod_\fm \Tbar_\fm \textnormal{ and }
 \TT_{\QQ_\ell} \cong \prod_\lambda \TT_{\QQ_\ell,\lambda^{(e)}} \cong  \prod_\lambda \TT_\lambda, $$ 
where $\fm$ runs through the maximal ideals of $\TT$ (and $\Tbar$) and $\lambda$ runs through the minimal primes of~$\TT$
(or, equivalently, all the prime=maximal ideals of~$\TT_{\QQ_\ell}$).
\end{proposition}

\subsubsection*{Package for computing $\ell$-adic decompositions}

The second author has developed the {\sc Magma} \cite{Magma} package {\sc pAdicAlgebras} (see \cite{pAdicAlgebras}) for computing
the objects appearing in Proposition~\ref{prop:commalg}.
The package depends on the second author's earlier {\sc Magma} package {\sc ArtinAlgebras} (see \cite{ArtinAlgebras}).

The main ingredients are standard linear algebra, especially over finite fields, and the algorithmic idempotent lifting
from Corollary~\ref{cor:alg-idem}.

\subsubsection*{Application of the commutative algebra to modular forms}

Let $S(\CC)$ be a space of modular forms, e.g.\ $S_k(\Gamma_1(N))$.
We only work with spaces that have a basis with coefficients in~$\ZZ$.
We denote by $S(R)$ the corresponding space with coefficients in the ring~$R$.
Here the notion $S(R)$ is the naive one via the standard $q$-expansion:
$S(R)$ is the set of $R$-linear combinations of the image of the $\ZZ$-basis in $R[[q]]$ via the standard $q$-expansion.
The space $S(R)$ can also be characterised as follows. The Hecke operators $T_n$ for $n \in \NN$ acting on $S(\CC)$
generate a ring (a $\ZZ$-algebra), denoted $\TT$, and we have the isomorphism
$$ S(R) \cong \Hom_\ZZ(\TT,R).$$
Concretely, if $\varphi \in \Hom_\ZZ(\TT,R)$, then $\sum_{n \ge 1} \varphi(T_n) q^n$ is a cusp form.
Thus a $\ZZ$-basis of $\TT$ gives rise to a `dual basis' of $S(R)$. We also speak of an `echelonised basis'.

By Proposition~\ref{prop:commalg}, we have the decompositions
$$ \TT_\QQ := \QQ \otimes_\ZZ \TT \cong \prod_{[f]} \TT_{[f]} \textnormal{ and } S(\QQ) \cong \bigoplus_{[f]} S_{[f]}(\QQ),$$
where the product and the sum run over $G_\QQ$-orbits of Hecke eigenforms.
If the space $S(\CC)$ is a newspace, then $S_{[f]}(\QQ)$ is the set of forms with coefficients in~$\QQ$
in the $\CC$-span of all the $G_\QQ$-conjugates of~$f$.
Concretely, $S_{[f]}(\ZZ)$ is the $\ZZ$-dual
of the $\ZZ$-algebra generated by the Hecke operators $T_n$ in $\TT_{[f]}$.
All Hecke operators acting on $S_{[f]}(\ZZ)$ are represented as matrices with $\ZZ$-entries.

We now consider $\TT_{\ZZ_\ell} = \ZZ_\ell \otimes_\ZZ \TT$.
Then we have $S(\ZZ_\ell) = \Hom_{\ZZ_\ell}(\TT_{\ZZ_\ell},\ZZ_\ell)$.
Importantly, again by Proposition~\ref{prop:commalg}, we have the decompositions
$$ \TT_{\ZZ_\ell} \cong \prod_{[\overline{f}]} \TT_{[\overline{f}]} \textnormal{ and } S(\ZZ_\ell) \cong \bigoplus_{[\overline{f}]} S_{[\overline{f}]}(\ZZ_\ell),$$
where the sum and the product run over the $G_{\FF_\ell}$-orbits of Hecke eigenforms in $S(\Fbar_\ell)$.
These correspond to the maximal ideals of $\TT_{\ZZ_\ell}$.
We refer to the $S_{[\overline{f}]}(\ZZ_\ell)$ either as $\ZZ_\ell$-orbits or as $G_{\FF_\ell}$-orbits.

We are also interested in $\QQ_\ell$-orbits of eigenforms inside a $\ZZ_\ell$-orbit.
By Proposition~\ref{prop:commalg}, $\QQ_\ell \otimes_{\ZZ_\ell} S(\ZZ_\ell) = S(\QQ_\ell)$
breaks as a direct sum
$$ S(\QQ_\ell) \cong \bigoplus_{[\tilde{f}]} S_{[\tilde{f}]}(\QQ_\ell),$$
where the sum runs over the $\Qbar_\ell$-valued eigenforms up to $G_{\QQ_\ell}$-conjugation.
The fact that these $G_{\QQ_\ell}$-orbits lie in a single $\ZZ_\ell$-orbit simply means that they are all congruent modulo a uniformiser.

\subsubsection*{Testing weak congruences}

The second author has developed the {\sc Magma} package {\sc WeakCong} (see \cite{WeakCong}), which has the purpose to
compute whether Hecke eigenforms over $\Qbar_\ell$ belong to given $\ZZ_\ell$-orbits
of Hecke eigenforms modulo powers of~$\ell$ (or uniformisers).
Here we briefly describe how it functions.

Let $n_1,\dots,n_r$ be indices such that $T_{n_1},\dots,T_{n_r}$ form a basis of the Hecke algebra~$\TT_{\ZZ_\ell}$
(which we may assume to be local by using the {\sc Magma} package {\sc pAdic\-Algebras}, see above).
We speak of {\em basis indices}.
These indices are computed via Nakayama's lemma, i.e.\ by reducing the matrices to $\FF_\ell$.

For any~$n$, we have $T_n = \sum_{i=1}^r a_{n,i} T_{n_i}$; in particular, $a_{n_j,i} = \delta_{i,j}$.
For each $i \in \{1,\dots,r\}$, we define a cusp form $f_i$ by specifying its coefficients as follows:
$$ a_n(f_i) := a_{n,i}.$$
Then $f_1,\dots,f_r$ form an $R$-basis of $\Hom_{\ZZ_\ell}(\TT_{\ZZ_\ell},R)$ for any $\ZZ_\ell$-algebra~$R$.
We call this basis {\em echelonised} because it is at the coefficients $n_1,\dots,n_r$.
It is the dual basis with respect to the basis $T_{n_1},\dots,T_{n_r}$ of~$\TT_{\ZZ_\ell}$.

Furthermore, we compute one $\Qbar_\ell$-eigenform for each $\QQ_\ell$-orbit inside the given $\ZZ_\ell$-orbit.
This is done via standard linear algebra over local fields, using both the new {\sc Magma} command LocalField and the older implementation.
If we find that a system of linear equations which mathematically must have a solution does not seem to have any, then
we lower the precision until the desired solution exists.
Thus, in this procedure generally some precision is lost.

Let $g = \sum_{n \ge 1} b_n q^n \in S(\Qbar_\ell)$ be an eigenform in some level and weight. Let $\cO$ be the valuation
ring of some finite extension of~$\QQ_\ell$ that contains all coefficients $b_n$ of~$g$, and let $\lambda$ be a uniformiser of~$\cO$.
The main purpose of this package is to compute the maximum integer~$m$ such that $g$ lies in a given
$\ZZ_\ell$-orbit (some level and some weight) modulo~$\lambda^m$.

Put $h := g - \sum_{i=1}^r b_{n_i} f_i$.
We then have:
$$ h \equiv 0 \mod \lambda^m \;\Leftrightarrow \; \exists\, s_1,\dots,s_r \in \cO: g \equiv \sum_{i=1}^r s_i f_i \mod \lambda^m.$$
This equivalence is clear as the basis is echelonised, whence automatically $s_i \equiv b_{n_i} \mod \lambda^m$ for all $i=1,\dots,r$.
The desired highest exponent~$m$ can thus be computed as the minimum of the valuations of the coefficients of~$h$ up to
the Sturm bound.

\section{Database of modular form orbits and higher congruences}\label{sec:database}

The first author has created a PostgreSQL database containing data on $\QQ$-, $\QQ_\ell$- and $\ZZ_\ell$-orbits,
as well as information on congruences modulo powers of~$\ell$.
We are currently planning to integrate parts of the database into the LMFDB.\footnote{\url{http://lmfdb.org}}

\subsubsection*{Technical features}

In this section we describe the way our database is organised and what kind of data it contains.
This will also highlight two important aspects of our approach:
\begin{itemize}
\item We do our best to avoid computing again data that are used more than once.
This aims to speed up the process of computing the $G_{\QQ_\ell}$-orbits. In order to do this we store a lot of useful information, even intermediate results, e.g.\ congruences with forms other than those that provide an optimal weight or level, even congruence of individual coefficients.
\item We try to parallelise as much of the problem as possible. This also aims at speeding up the computation of congruences. This becomes especially handy when the coefficient fields of the forms that are compared become large.
\end{itemize}
We will come back to both of these features after the description of the database tables.
We list them together with a brief description of the data each one holds.
\begin{enumerate}
\item {\bf Modular form spaces over $\QQ$}: For every level and weight we store some useful information: The dimension of its Eiseinstein subspace, old cuspidal subspace, new cuspidal subspace as well as the number of new Eisenstein $\QQ$-Galois orbits and the number of newform $\QQ$-Galois orbits.
\item {\bf Bases of modular form spaces over $\QQ$}: Here we store the basis in terms of modular symbols for every space in the previous table. This in Magma readable format.
\item {\bf Eigenforms over $\QQ$}: For every space over $\QQ$, we store an entry for every Eisenstein and newform $\QQ$-Galois orbit uniquely determined my its level, weight and orbit number.
\item {\bf Hecke matrices over $\ZZ$}: For each of the newform orbits in the previous table we store a list (up to a bound that can be increased as needed) of all the Hecke matrices acting on the $\QQ$-subspace spanned by this orbit.
\item {\bf Lattices}: For each of the newform orbits in the $\QQ$-eigenforms table, we store a list of base change matrices that ensure the matrices in the table above, after base change, are with respect to the same basis.
\item {\bf $\ell$-adic idempotents}: Given a newform from the $\QQ$-eigenforms list and a prime number $\ell$, we store a list of idempotents which provide the decomposition of the corresponding $\ell$-adic Hecke algebra into local factors (see Proposition \ref{prop:commalg}), their number and the $\ell$-adic precision that they were computed in.
\item {\bf $\FF_\ell$-Galois orbits}: For each entry in the table above (i.e. a list of idempotents), we store an $\ZZ$-integral basis for each of the components (indexed by the idempotents in this list) that the parent $\QQ$-Galois orbit of newforms breaks into.
\item {\bf $\QQ_\ell$-Galois orbits of newforms}: For each $\QQ$-Galois orbit of newforms and the prime $\ell$, we store the $\QQ_\ell$-Galois orbit of newforms it decomposes into, along with the $\ell$-adic precision they were computed in.
\end{enumerate}

These are the tables that provide a hierarchical organisation of the objects involved in the database and we tried to present it in a top to bottom fashion were an entry in one of these table will be associated with many entries in the ones mentioned after it. 

There are some auxiliary tables where all the congruence information is stored. We store everything down to congruences of individual pairs of coefficients. These are detailed catalogs of all meaningful congruences when it comes to level or weight lowering, weak or strong.

It is obvious that the comparison of two eigenvalues at a prime $p$ is independent from the comparison of the ones at some other prime~$q$.
We thus run a multi-threaded application utilising as many CPU cores as possible where all threads
compare a specific pair of eigenvalues each simultaneously. 
Let us stress here that the design of the database and the multi-threaded application is such that it allows us to utilise more than one
server and/or personal computers to compute even more congruences simultaneously.
Extra care has been taken to avoid overlapping of threads,
i.e.\ two of those computing the same congruence, but we choose not to elaborate on these technical matters.

The current size of the database is 488GB. It contains 3906 $\QQ$-eigenforms, of level and weight up to 361 and 298 respectively (not of all possible combinations of course).

\subsubsection*{Accessibility}

We have designed a basic web interface\footnote{\url{http://math.uni.lu/~tsaknias/elladicdatabase_2.php}}
for the database which currently allows one to query the database about the following:
\begin{enumerate}
\item Given a $G_\QQ$-orbit $[f]$ and a prime $\ell$, return  $G_{\QQ_\ell}$-orbits appearing in it.
\item Given a $G_\QQ$ orbit $[f]$, a prime $\ell$ and a positive integer $n$, return the $G_{\QQ_\ell}$-orbits
that are congruent to the ones corresponding to $[f]$ and $\ell$ modulo $\ell^n$ and are of the smallest weight possible,
i.e.\ the answer to the strong weight lowering modulo $\ell^n$ problem for~$[f]$.
\item Given a $G_\QQ$-orbit $[f]$ and a prime $\ell$, return a list of downloadable files (one for each $G_{\QQ_\ell}$-orbit)
containing all the $\ell$-adic, prime-indexed Hecke polynomials (that are stored in the database) for each $G_{\QQ_\ell}$-orbit.  
\end{enumerate}

\subsubsection*{Some remarks on the algorithms used}

We now describe how we computed the various orbits.
Our algorithm is implemented in the {\sc Magma} computer algebra system \cite{Magma}.
Assume as input a given level $N$, weight $k$ and prime~$\ell$.

\begin{enumerate}
\item Compute the newsubspace of the cuspidal subspace of the modular symbols of level $N$ and weight $k$. Decompose this subspace into irreducible Hecke modules. These correspond to $G_\QQ$-orbits. This is done with standard {\sc Magma} commands.

\item For a given irreducible Hecke module of the previous decomposition, compute the matrices for all operators $T_n$ acting on it up to a sufficient bound~$B$.

\item Use the package {\sc pAdicAlgebras} \cite{pAdicAlgebras} to factor the completion of the Hecke algebra at~$\ell$
into local factors over~$\ZZ_\ell$. Each of these factors corresponds to a $G_{\FF_\ell}$-orbit.
Project the matrices representing the $T_n$'s onto each of these local factors.

\item After tensoring with $\QQ_\ell$, each of these $G_{\FF_\ell}$-orbits is the sum of all the $G_{\QQ_\ell}$-orbits admitting the same reduction mod $\ell$.
For each such orbit, take the collection of projections of the Hecke matrices onto it computed in the previous step and decompose
the corresponding $\QQ_\ell$-vector space into simultaneous generalised eigenspaces by applying each operator successively.
The resulting decomposition is the breaking of the corresponding $G_{\FF_\ell}$-orbit into the $G_{\QQ_\ell}$-ones that coincide mod $\ell$.
\end{enumerate}

\bibliography{References}
\bibliographystyle{alpha}

\end{document}